\documentclass[11pt]{amsart}
\usepackage{amssymb}
\usepackage{palatino}
\usepackage{tikz-cd}
\input amssym.def

\usepackage{xcolor}
\usepackage{amsmath, amsfonts}
\usepackage{amssymb}
\usepackage{amscd}
\usepackage[mathscr]{eucal}
\usepackage{palatino}
\usepackage{theoremref}
\usepackage{tikz}
\usepackage{float}
\usepackage{comment}
\usepackage{dsfont}
\usepackage{hyperref}
\hypersetup{colorlinks,citecolor= black,linkcolor= black,urlcolor=black}

\usepackage{xspace}
\usepackage{verbatim}

\newfont{\cyrr}{wncyr10}

\newcommand{\thmref}[1]{Theorem~\ref{#1}}

\newcommand{\lemref}[1]{Lemma~\ref{#1}}

\newcommand{\rmkref}[1]{Remark~\ref{#1}}

\newtheorem{thm}{Theorem}

\newtheorem{lem}[thm]{Lemma}

\newtheorem{rmk}{Remark}[section]

\newcommand{\Z}{{\mathbb Z}}

\def\({\left(}
\def\){\right)}
\def\[{\left[}
\def\]{\right]}

\def\GL{{\rm GL}}
\def\G{{\rm G}}

\def\N{\mathbb{N}}
\def\R{\mathbb{R}}
\def\C{\mathbb{C}}
\def\Q{\mathbb{Q}}

\def\cO{\mathcal{O}}

\def\cI{\mathcal{I}}
\def\cJ{\mathcal{J}}
\def\cQ{\mathcal{Q}}
\def\cL{\mathcal{L}}
\def\cE{\mathcal{E}}

\def\fp{\mathfrak{p}}

\def\e{\epsilon}

\usepackage{xcolor}
\newcommand{\red}[1]{\textcolor{red}{#1}}

\title[Largest prime factor of Fourier coefficients in short intervals]{A note on
Fourier coefficients of Hecke eigenforms in short intervals}

\author{Sanoli Gun  and Sunil Naik}
\address{Sanoli Gun \newline
The Institute of Mathematical Sciences, 
A CI of Homi Bhabha National Institute, 
CIT Campus, Taramani, 
Chennai 600 113, 
India.}
\email{sanoli@imsc.res.in}

\address{Sunil L Naik \newline
Department of Mathematics,
Queen's University, Jeffery Hall, 
99 University Avenue, 
Kingston, ON K7L 3N6, 
Canada.}
\email{naik.s@queensu.ca}

\begin{document}
	
\hfuzz 5pt
	

\keywords{Large prime factors of Fourier coefficients of Hecke eigenforms, 
Explicit version of Chebotarev density theorem
in short intervals,  number of non-zero Fourier coefficients at primes in short intervals}	
	
\maketitle
	
\begin{abstract}
In this article, we investigate large prime factors of Fourier coefficients of 
non-CM normalized cuspidal Hecke eigenforms in short intervals.
One of the new ingredients involves deriving an explicit version of 
Chebotarev density theorem in an interval of length $\frac{x}{(\log x)^A}$
for any $A>0$, modifying an earlier work of Balog and Ono. Furthermore, we
need to strengthen a work of Rouse-Thorner to
derive a lower bound for the largest prime factor 
of Fourier coefficients in an interval of 
length $x^{1/2 + \epsilon}$ for any $\epsilon >0$. 
\end{abstract}

\section{Introduction and Statements of Results}
Let $x, y$ be real numbers, $p, q$ be prime numbers, $N \ge 1$ be an integer
and $f$ be a non-CM normalized cuspidal
Hecke eigenform of weight $k \ge 2$ for $\Gamma_0(N)$ 
with integer Fourier coefficients $a_f(m)$ for $m \geq 1$. 
 In this article, we investigate large prime factors of Fourier coefficients of $f$  in short
 intervals. We note that even the existence of a prime $p$ in short intervals
 with $a_f(p) \ne 0$ in itself a difficult question.  
 It follows from a recent work of Lemke Oliver and Thorner \cite[Theorem 1.6]{OT}
that there exists an absolute constant $\delta>0$ and a prime $p \in (x, x +  y]$  
such that $a_f(p) \neq 0$ when $y \ge x^{1- \delta}$.

In this work, we find prime factors of size at least
 $(\log x)^{1/8}$ in intervals of length $\frac{x}{{(\log x})^A}$ for
 any positive $A$.  This begs the question about the
 possible/expected order of such prime factors in such short or even shorter
 intervals of size/length, say, a small power of $x$.
 We show that under the generalized Riemann hypothesis for 
 all symmetric power $L$-functions of $f$ and all Artin $L$-series, 
 one can  find  prime factors of size at
 least $x^{\epsilon/7}$ in intervals of length $x^{1/2 + \epsilon}$ for any 
 $\epsilon < 1/10$.

 In an earlier work \cite{BGN}, the present authors 
along with Bilu investigated lower bounds for the
largest prime factor of $a_f(p)$.
However finding such large prime factors in short intervals  
is a different ball game. We  need to 
establish a explicit version of a result of Balog-Ono \cite{BO}.
Further, for the conditional result on 
the generalized Riemann hypothesis (as specified above), 
we need  to strengthen a conditional result of Rouse-Thorner \cite{RT} 
(see also Thorner \cite{Th}) in short intervals.

Before proceeding further, let us fix a notation.
For any integer $n$, let $P(n)$ denote the largest prime factor 
of $n$ with the convention that $P(0)= P(\pm 1) = 1$. 
Throughout the article, by GRH, 
we mean the generalized Riemann hypothesis 
for all symmetric power $L$-functions of $f$ 
and all Artin $L$-series, unless otherwise specified.
In this set up, we prove the following results.

\begin{thm}\label{thm1}
Let $f$ be a non-CM normalized cuspidal Hecke eigenform 
of weight $k$ for $\Gamma_0(N)$ 
having integer Fourier coefficients $a_f(m)$ for $m \ge 1$.
For positive real numbers $A, \e$ and natural numbers $n \ge 1$, 
there exists a prime $p \in (x,  ~x + \frac{x}{(\log x)^A}]$ such that
$$
P(a_f(p^{n})) ~>~ (\log x^n)^{1/8} (\log\log x^n)^{3/8 -\e}
$$
for all sufficiently large $x$ depending on $A, \e, n$ and $f$.
\end{thm}

\begin{rmk}\label{density}
The lower bound in \thmref{thm1} can be replaced 
by $(\log x^n)^{1/8}(\log\log x^n)^{3/8}u(x^n)$ 
for any real valued non-negative function $u$ 
with $u(x) \to 0$ as $x \to \infty$.
\end{rmk}

\begin{thm}\label{thm2}
Suppose that GRH is true, $f$ is as in \thmref{thm1} and $\e \in (0, \frac{1}{10})$.
For any natural number $n >1$, there exists a positive real constant $c$ (depending on $\e, n, f$),
a positive constant $b$ (depending on $n$) and a prime number $p \in (x, x+ x^{\frac{1}{2}+\e}]$
such that
$$
P\(a_f(p^{n})\) ~>~ c x^{\e b}
$$
for all sufficiently large $x$ depending on $\e, n, f$. When $n=1$, there
exists a positive real constant $c$ (depending on $\e, f$)
and a prime number $p \in (x, x+ x^{\frac{1}{2}+\e}]$
$$
P\( a_f(p) \) 
~>~
c x^{\e/7} (\log x)^{2/7}
$$
for all sufficiently large $x$ depending on $\e, f$. 
\end{thm}

If we are allowed to go up to a little longer than  $x^{\frac{3}{4}}$, then GRH
ensures even  larger prime factors. More precisely, we have the following:

\begin{thm}\label{thmn}
Suppose that GRH is true and 
let $\eta(x) = x^{3/4} \log x \cdot \log\log x$. Then
for all  $x$ is sufficiently large
(depending on $n$ and $f$),  there exists 
a prime $p \in (x,  x + \eta(x) ]$  such that 
$$
P(a_f(p^{n})) ~>~ c x^{1/28} (\log x)^{3/7} (\log\log x)^{1/7}
$$
for some positive real number $c$ depending on $f$. 
\end{thm}

\begin{rmk}\label{rmk2}
 Suitable modifications of the proofs of \thmref{thm1}, \thmref{thm2}
 and \thmref{thmn}
 will show that these theorems are true for a set of primes of positive
 density. More precisely, it follows that the number of primes 
 $p \in (x,  ~x + \frac{x}{(\log x)^A}]$ 
for which \thmref{thm1} is true is at least $\frac{a_1x}{(\log x)^{A+1}}$  
for some positive constant $a_1$ and  for all sufficiently large $x$. 
If $\e > 0$ is sufficiently small, then the number of primes 
$p \in (x, x+ x^{\frac{1}{2}+\e}]$ for which \thmref{thm2} is true 
is at least $a_2\frac{x^{1/2+ \e}}{\log x}$ for some $a_2>0$ and 
for all sufficiently large $x$. Further, the number of primes 
$p \in (x,  x + \eta(x) ]$ for which \thmref{thmn} is true 
is at least $\frac{a_3\eta(x)}{\log x}$ 
for some positive constant $a_3$ and for all sufficiently large $x$.
\end{rmk}

\smallskip

\section{Preliminaries}

\medskip

\subsection{Distribution of zeros of Dedekind zeta functions}
Let $L/K$ be an abelian extension of number fields 
with Galois group $G$. Then we have
$$
\zeta_L(s)  ~=~ \prod_{\chi} L(s, \chi, L/K),
$$
where $\chi$ runs over the irreducible characters of $G$ 
(see \cite[Ch. XII]{La}, \cite[VII]{Ne} for more details). 
Let $\mathfrak{f}_\chi$ denote the conductor of $\chi$ 
and set
$$
\cQ ~=~ \cQ(L/K)~=~ \max_{\chi} N_K\(\mathfrak{f}_\chi\),
$$
where $N_K$ denotes the absolute norm on $K$.
Also let
$$
Q ~=~ Q(L/K) ~=~ D_K \cQ n_K^{n_K},
$$
where $D_K$ is the absolute discriminant of $K$ and $n_K = [K:\Q]$. 
We write $s \in \C$ as $s= \sigma+it$, 
where $\sigma= \Re(s)$ and $t= \Im(s)$.
A zero-free region of $\zeta_L(s)$ is given by 
the following theorem 
(\cite[Theorem 3.1]{TZ}, see also \cite[Theorem 1.9]{We}).
\begin{thm}\label{zerofreez}
There exists an absolute positive constant $c_1$ 
such that the Dedekind zeta function $\zeta_L(s)$ 
has atmost one zero in the region
$$
\sigma  ~>~ 1- \frac{c_1}{\log \(Q (|t|+3)^{n_K}\)}.
$$	
Suppose such  a zero $\beta_1$ exists, then it is real, 
simple and is a zero of the L-function corresponding 
to a real Hecke character $\chi_1$ of $G$.
\end{thm}

\begin{rmk}
The above exceptional zero $\beta_1$ (if it exists) 
is usually known as  Landau-Siegel zero.
\end{rmk}
For $0 \leq \sigma \leq 1$ and $T \geq 1$, 
let
$$
N(\sigma, T, \chi) 
~=~ 
\#\{ \rho= \beta+i \gamma ~:~ L(\rho, \chi, L/K)= 0, 
~ \sigma < \beta < 1 \text{ and } -T <\gamma < T\},
$$
where the zeros $\rho$ are counted with multiplicity.
Set
$$
N(\sigma, T) 
~=~
 \sum_{\chi} N(\sigma, T, \chi),
$$
where $\chi$ runs over the irreducible characters of $G$.
In this set up, we have the following theorem 
(see \cite[Theorem 3.2]{TZ},  \cite[Theorem 4.3]{We}).
\begin{thm}\label{ZerDen}
There exists an absolute constant $c_2 \geq 1$ 
such that
$$
N(\sigma, T) 
~\ll~ 
B_1 \(QT^{n_K}\)^{c_2 (1-\sigma)}
$$
uniformly for any $ 0 < \sigma < 1$ and $T \geq 1$. 
Here 
$$
B_1 ~=~ B_1(T) 
~=~ \min\{1, (1-\beta_1) \log (QT^{n_K})\}.
$$
\end{thm}

\subsection{Chebotarev density theorem in short intervals}
Let $L/K$ be a Galois extension of number fields 
with Galois group $G$. Let $n_L= [L : \Q]$ and $n_K = [K:\Q]$. 
Also let $D_L$ (resp. $D_K$) denote the absolute discriminant 
of $L$ (resp. K). 
For a conjugacy class $C \subseteq G$, 
define  
$$
\pi_C(x, L/K) 
~=~ 
\#\{\fp  \subseteq \cO_K ~:~ N_K(\fp) \leq x,~ 
\fp \text{ is unramified in } L 
\text{ and } [\sigma_\fp] = C\},
$$
where $\sigma_\fp$ is a Frobenius element of $\fp$ in $G$ 
and $[\sigma_\fp]$ denotes the conjugacy class of $\sigma_\fp$ in $G$.
In \cite{BO}, Balog and Ono  proved the following theorem.
\begin{thm}\label{B0-Cheb}
Let $\e >0$ be a real number 
and $x^{1 -1/c(L)  + \e} \leq y \leq x$, 
then we have
$$
(1-\e) \frac{\#C}{\#G} \frac{y}{\log x} 
~<~ 
\pi_C\(x+y, L/K\) - \pi_C\(x, L/K\) 
~<~
 (1+\e) \frac{\#C}{\#G} \frac{y}{\log x}
$$
for all sufficiently large $x$ depending on $\e$ and $L$. 
Here
$$
c(L) ~=~
\begin{cases}
& n_L \phantom{mm}\text{if } n_L \geq 3, \\
& \frac{8}{3} \phantom{mm}\text{if } n_L = 2, \\
& \frac{12}{5}  \phantom{mm}\text{if } n_L = 1.
\end{cases}
$$ 
\end{thm}
For our application, we need a version of \thmref{B0-Cheb} 
which is uniform in $L$. 
In Section \ref{proofsec}, we prove the following explicit version 
of the Chebotarev density theorem in short intervals.

\begin{thm}\label{Cheb-short}
There exists a positive absolute constant $c_3$ such that 
if $y \geq x^{1-c_3/n_L}$ and 
$\log x \gg_{c_3} \log \(D_L n_L^{n_L})\)$, 
then we have	
$$
\bigg| 
\pi_C\(x+y, L/K\) - \pi_C\(x, L/K\)  
~-~
\frac{\#C}{\#G} 
\(\frac{y}{\log x} - \theta_1 \frac{(x+y)^{\beta_1} 
- x^{\beta_1}}{\beta_1 \log x}\)\bigg| 
~\leq~
 \frac{1}{4} \frac{\#C}{\#G} \frac{y}{\log x}.
$$
Here $\theta_1 \in \{-1, 1\}$ if the Landau-Siegel zero of 
the Dedekind zeta function $\zeta_L(s)$ exists 
and $\theta_1 = 0$ otherwise.
\end{thm}

\begin{rmk}
The constant $\frac{1}{4}$ in \thmref{Cheb-short} 
can be replaced with any small positive real number (see Section 3.1).
\end{rmk}

\subsection{Hecke eigenforms and $\ell$-adic Galois representation}\label{Smod}
Let $f$ be as in section 1 and $m$ be a positive integer. 
For any integer $d > 1$ and real number $x >0$, 
let
\begin{equation*}
\begin{split}
&\pi_{f,m}(x,d) 
~=~
\#\{p \leq x : a_f(p^m) \equiv 0 ~(\text{ mod } d) \}.
\end{split}
\end{equation*}
Let $\text{Gal}(\overline{\Q}/\Q)$ be the Galois group of $\overline{\Q}/\Q$
and for a prime $\ell$, let $\Z_\ell$ 
denote the ring of $\ell$-adic integers. 
By the work of Deligne \cite{De}, 
there exists a continuous representation
\begin{equation*}
\rho_{d} ~:~ 
{\text{Gal}(\overline{\Q}/\Q)} ~\rightarrow~
{\GL}_2\(\prod_{\ell | d} \Z_\ell\)
\end{equation*}
which is unramified outside the primes dividing $dN$. 
Further, if $p \nmid dN$, then we have
$$
\text{tr}\rho_{d}(\sigma_p) 
~=~ a_f(p) 
\phantom{mm}\text{and}\phantom{mm} 
\text{det}\rho_{d}(\sigma_p) 
~=~ p^{k-1},
$$
where $\sigma_p$ is a Frobenius element of $p$ in $\text{Gal}(\overline{\Q}/\Q)$. 
Here $\Z$ is embedded 
diagonally in $\prod_{\ell | d} \Z_\ell$.
Let $\tilde{\rho}_{d}$ denote the reduction of $\rho_{d}$ modulo $d$ : 
\begin{equation*}
\tilde{\rho}_{d} ~:~ {\text{Gal}(\overline{\Q}/\Q)} 
~\xrightarrow[]{\rho_{d}}~
{\GL}_2\(\prod_{\ell | d} \Z_\ell\)
~\twoheadrightarrow~
{\GL}_2(\Z/ d\Z).
\end{equation*}
Also denote by $\tilde{\rho}_{d,m}$, the composition of 
$\tilde{\rho}_{d}$ with $Sym^m$, 
where $Sym^m$ denotes the symmetric $m$-th power map :
\begin{equation*}
\tilde{\rho}_{d,m} ~:~  {\text{Gal}(\overline{\Q}/\Q)} 
~\xrightarrow[]{\rho_{d}}~ 
{\GL}_2\(\prod_{\ell | d} \Z_\ell\)  
~\rightarrow{}~
{\GL}_2(\Z/ d\Z)  
~\xrightarrow[]{Sym^m}~
{\GL}_{m+1}(\Z/ d\Z).
\end{equation*}
For $p \nmid dN$, we have
\begin{equation*}
\text{tr}\tilde{\rho}_{d,m}(\sigma_{p}) ~=~ a_f(p^m)~ (\text{mod } d). 
\end{equation*}
Let $H_{d,m}$ be the kernel of $\tilde{\rho}_{d,m}$, 
$K_{d,m}$ be the subfield of $\overline{\Q}$ fixed by $H_{d,m}$ 
and 
$$
{\G}_{d,m} = \text{Gal}(K_{d,m}/\Q) \cong \text{Im}(\tilde{\rho}_{d, m}).
$$ 
Suppose that $C_{d,m}$ is the subset of 
$\tilde{\rho}_{d,m}(\text{Gal}(\overline{\Q}/\Q))$ consisting of elements of trace zero. 
Let us set $\delta_{m}(d) = \frac{|C_{d,m}|}{|{\G}_{d,m}|}$.
For any prime $p\nmid dN$, the condition 
$a_f(p^m) \equiv 0~ (\text{mod }d)$ is equivalent to
the fact that $\tilde{\rho}_{d,m}(\sigma_p) \in C_{d,m}$, 
where $\sigma_{p}$ is a Frobenius element of $p$ in $\text{Gal}(\overline{\Q}/\Q)$. 
Hence by the Chebotarev density theorem applied to $K_{d,m} / \Q$, 
we have
\begin{equation*}
\lim_{x \to \infty} \frac{\pi_{f,m}(x,d)}{\pi(x)}
~=~
\frac{|C_{d,m}|}{|{\G}_{d,m}|}
~=~ \delta_{m}(d).
\end{equation*}
Applying \thmref{Cheb-short}, we can now deduce the following result.
\begin{thm}\label{pifshort}
Let $f$ be a non-CM normalized cuspidal Hecke eigenform
of weight $k$ and level $N$ with integer Fourier coefficients 
$a_f(n)$ for $n \ge 1$. 
Then there exists a positive absolute constant $c_3$ 
such that if $y \geq x^{1- \frac{c_3}{d^4}}$ and 
$\log x \gg_{c_3} d^4 \log (dN)$, then 
$$
\pi_{f, m}\(x+y, d\) - \pi_{f, m}\(x, d\) 
~\ll~ 
\delta_m(d) \frac{y}{\log x}.
$$
\end{thm}
When $m=1$, we have the following result 
(see \cite[Proof of Theorem 3]{GM}, 
\cite[Lemma 5.4]{MMprime}, \cite[Section 4]{SeCh}).
\begin{lem}\label{deltal}
For any prime $\ell$, we have
$$
\delta(\ell) ~=~ \frac{1}{\ell} +O\(\frac{1}{\ell^2}\) 
\phantom{mm}\text{and}\phantom{mm}
\delta(\ell^n) ~=~O\(\frac{1}{\ell^n}\)
$$
for any $n \in \N$. Here $\delta(\ell) = \delta_1(\ell)$.
\end{lem}

When $m+1$ is an odd prime $q$, the present authors 
in an earlier work (see \cite[Lemma 17, Lemma 18]{GN})
proved the following results.

\begin{lem}\label{deq-1l}
Let $q, \ell$ be primes with $q$ odd. Then 
$\delta_{q-1}(\ell) = 0$ unless 
$\ell \equiv 0, \pm 1 ~(\text{mod } q)$
and 
$$
\delta_{q-1}(\ell) ~\ll~ \frac{q}{\ell},
$$
where the implied constant depends only on $f$. Also we have
\begin{equation*}
\delta_{q-1}(\ell) 
~=~
\begin{cases}
& \frac{q-1}{2} \frac{1}{\ell-1}, 
\phantom{mm}\text{if}\phantom{mm} 
\ell \equiv 1~ (\text{ mod } q) \\
& \frac{q-1}{2}   \frac{1}{\ell+1}, 
\phantom{mm}\text{if}\phantom{mm} 
\ell \equiv -1~ (\text{ mod } q) \\	
& \frac{q}{q^2 -1}, \phantom{mmm}\text{if} \phantom{mm} \ell = q 
\end{cases}
\end{equation*}
for all sufficiently large $\ell$.
\end{lem}

\begin{lem}\label{dq-1lm}
For any integer $n \geq 2$ and primes $\ell, q$ with $q$ odd, we have
$$
\delta_{q-1}(\ell^n) 
~\ll~
\frac{1}{\ell^{n-1}} \delta_{q-1}(\ell),
$$
where the implied constant depends only on $f$.
We also have
$$
\delta_{q-1}(\ell^n) ~=~ \frac{1}{\ell^{n-1}} ~\delta_{q-1}(\ell)
$$
if $\ell \neq q$ and $\ell$ is sufficiently large.
Further $\delta_{q-1}(q^n) = 0$ for $q \geq 5$.
\end{lem}

Conditionally under GRH, i.e. assuming the generalized Riemann hypothesis 
for all Artin L-series, we can deduce the following theorem
 by applying a result of Lagarias and Odlyzko \cite[Theorem 1.1]{LO} 
 (see also \cite[Lemma 5.3]{MMprime}).
\begin{thm}\label{pifGRH}
	Suppose that GRH is true and $f$ is a non-CM form. 
	Then we have
	$$
	\pi_{f, m}(x, d) 
	~=~ \delta_m(d) \(\pi(x) ~+~ O\( x^{1/2} d^4 \log (dNx)\)\) 
	~+~ O\(d^4 \log (dN)\).
	$$
\end{thm}

\bigskip

\subsection{Sato-Tate conjecture in short intervals}\label{ST}
Let $f$ be as before and 
$$
\lambda_f(p) ~=~ \frac{ a_f(p)}{2p^{(k-1)/2}}.
$$ 
The Sato-Tate conjecture states that the numbers $\lambda_f(p)$
are equidistributed in the interval ${[-1, 1]}$ with respect 
to the Sato-Tate measure 
$$
d\mu_{ST} = (2/\pi)\sqrt{1- t^2}~dt.
$$ 
This means that for any $-1 \leq a \leq b \leq 1$, 
the density of the set of primes~$p$ 
satisfying ${\lambda_f(p) \in [a, b]}$ is 
\begin{equation*}
\frac{2}{\pi}\int_a^b\sqrt{1-t^2}~ dt.
\end{equation*}
It is now a theorem due to the 
works of Barnet-Lamb, Clozel, Geraghty, 
Harris, Shepherd-Barron and Taylor 
(\cite[Theorem~B]{BGHT}, \cite{{CHT}, {HST}}).

We will need Sato-Tate conjecture in short intervals
due to Lemke Oliver and Thorner.  For this,
we need to introduce Chebyshev polynomials.  
The Chebyshev polynomials of second kind are defined by
\begin{equation*}
U_0(x) ~=~ 1,~~ U_1(x) ~=~ 2x 
\phantom{m}\text{and}\phantom{m} 
U_n(x) ~=~ 2x U_{n-1}(x) - U_{n-2}(x) ~ \text{ for  }n \geq 2.
\end{equation*}
The generating function of $U_n$ is given by
$$
\sum_{n=0}^{\infty} U_n(x) t^n 
~=~ \frac{1}{1-2tx+t^2}.
$$
Note that if $p \nmid N$, then $U_n(\lambda_f(p))$ 
is the  Dirichlet coefficient of $ L \( s, \text{Sym}^n \pi_f \)$ at $p$, 
where $\pi_f$ denotes the cuspidal representation 
of $\text{GL}_2\(\mathbb{A}_\Q\)$ attached  to $f$.
Let $M$ be a natural number. 
A subset $I \subseteq [-1, 1]$ is said to be $\text{Sym}^M$-minorized 
if there exist constants $b_0, b_1, \cdots , b_M \in \R$ with $b_0 > 0$ 
such that
$$
\mathds{1}_{I}(t) 
~\geq~ 
\sum_{n=0}^{M} b_n U_n(t)
~~\text{ for all~} t \in [-1,1].
$$
Here $\mathds{1}_I$ denotes the indicator function of $I$.

\begin{rmk}\label{SymMin}
Let $B_0 = \frac{1+\sqrt{7}}{6} = 0.6076 \cdots$ and 
$B_1 =\frac{-1+\sqrt{7}}{6} = 0.2742 \cdots$. 
Then the interval $[-1, b]$ can be $\text{Sym}^4$-minorized 
if $b > -B_0$ and $[a, 1]$ can be $\text{Sym}^4$-minorized 
if $a \in [B_1, B_0)$. It can be shown that the interval $I= [-1, -0.1]$ is
$\text{Sym}^4$-minorized with $b_0 > 0.08$ (see \cite[Lemma A.1]{OT}).
Further, any interval $I \subseteq [-1, 1]$ can be $\text{Sym}^M$-minorized
if $M$ is sufficiently large (see~\cite[Page 6997, Remark 1]{OT}).
\end{rmk}

In this context, Lemke Oliver and Thorner 
proved the following version of the Sato-Tate conjecture 
in short intervals (see \cite[Thorem 1.6]{OT}).
\begin{thm}\label{STshort}
Let $f$ be a non-CM normalized  Hecke eigenform 
of weight $k$ and level $N$. Also let $I \subseteq [-1, 1]$ 
be a subset which can be $\text{Sym}^M$-minorized.
Then there exists a constant $c_4 \in (0, 1)$ 
depending on $I$ and $N$ such that
if $y \geq x^{1-c_4}$, then
$$
\sum_{\substack{x < p < x+y \\ p  \nmid N}} 
\mathds{1}_I\(\lambda_f(p)\) \log p 
~\asymp~ y
$$
for all sufficiently large $x$ depending on $f$ and $M$. 
Here the implied constant depends on $I$ and $M$.
\end{thm}

Conditionally under GRH, i.e, assuming the generalized Riemann hypothesis 
for all symmetric power L-functions $L(s, \text{Sym}^m \pi_f)$, we have 
the following theorem due to Rouse and Thorner (see \cite{RT}, \cite{Th}).

\begin{thm}\label{EfSTGRH}
Suppose that GRH is true and $f$ is a non-CM form. 
Also let $I \subseteq [-1, 1]$ be an interval. Then we have
\begin{equation}\label{error}
\#\{p \leq x ~:~ p \nmid N,~ \lambda_f(p) \in I\} 
~=~ 	\mu_{ST}(I) \pi(x)  ~+~
O\(x^{3/4} \frac{\log \(kNx\)}{\log x}\).
\end{equation}
\end{thm}

\begin{rmk}
As remarked by Thorner in \cite{Th}, it is expected that the error term in \eqref{error} 
can be replaced by $O(x^{1/2+\e})$ for any $\e > 0$, 
where the implied constant will depend on $\e$ and $f$.
\end{rmk}

Let $\e>0$ be a real number. From \thmref{EfSTGRH}, 
it follows that if $y \geq x^{3/4} \log x \log\log x$, then
$$
\sum_{\substack{x < p \leq x+y \\ p \nmid N}} 
\mathds{1}_I \( \lambda_f(p)\) \log p 
~\ge~	 
\(	\mu_{ST}(I) -\e \)  y
$$
for all sufficiently large $x$ depending on $\e, I$ and $f$.
In section 4, we will prove the following theorem 
conditionally under the generalized Riemann hypothesis
for all symmetric power $L$-functions of $f$.
\begin{thm}\label{STshomin}
Suppose that GRH is true, $f$ is a non-CM form and $\e > 0$ is a real number. 
Let $I \subseteq [-1, 1]$ be a subset which can be $\text{Sym}^M$-minorized
and $b_0$ be as before. Then for $y \geq x^{1/2} (\log x)^3$, we have
$$
\sum_{\substack{x < p \leq x+y \\ p \nmid N}} 
\mathds{1}_I \( \lambda_f(p)\) \log p 
~\geq~ 
\(b_0 -\e \) y  
$$
for all sufficiently large $x$ depending on $\e, I, M$ and $f$.
\end{thm}

\section{Chebotarev density theorem in short intervals}\label{proofsec}

\smallskip

\subsection{Proof of \thmref{Cheb-short}}
Let the notations be as in section 2 and define
$$
\Psi_C(x, L/K) 
~=~ \sum_{\substack{N_K(\fp)^m \leq x 
\\ \fp~ \text{ unramified } \\ [\sigma_{\fp}]^m = C}} 
\log N_K(\fp).
$$
Let $g$ be a non-identity element of $G$,
$H = <g>$ and $E = L^H$. 
Also let $x \geq 2$, $T \geq 2$ and $1 \leq y \leq x$. 
Then from \cite[Theorem 7.1]{LO}, we get
\begin{equation}\label{PsiSh1}
\Psi_C(x+y, L/K) - \Psi_C(x, L/K) 
~=~
\frac{\#C}{\#G} 
\( y - \sum_{\chi} \overline{\chi}(g) 
\sum_{\substack{\rho \\ |\gamma| < T}} 
\frac{(x+y)^{\rho} - x^\rho}{\rho} \) 
~+~  \cE_1 ~+~ \cE_2,
\end{equation}
where $\chi$ runs over irreducible characters of $H$ and 
$\rho$ runs over non-trivial zeros of $L(s, \chi, L/E)$. 
Further, we have
\begin{equation}\label{E1}
\cE_1 ~\ll~  
\frac{\#C}{\#G} \( \frac{x \log x + T}{T} \log D_L 
~+~ n_L \log x 
~+~ \frac{n_L x \log x \log T}{T}\)
\end{equation}
and 
\begin{equation}\label{E2}
\cE_2 ~\ll~ \log x \log D_L ~+~ n_K\frac{x \log^2 x}{T} .
\end{equation}
Let us set $\cL = \log\(Q T^{n_E}\)$, 
where $Q = D_E \cQ(L/E) n_E^{n_E}$ (see section 2.1).

We estimate the above double sum over $ \chi$ and $\rho$ as follows:
\begin{equation*}
\begin{split}
\Bigg|
\sum_{\chi} \overline{\chi}(g) \sum_{\substack{\rho \neq \beta_1 \\ |\gamma| < T}} 
\frac{(x+y)^\rho - x^\rho}{\rho}
\Bigg|
&~\leq~
\sum_{\substack{\chi, ~~\rho \neq \beta_1  \\ |\gamma| < T
	\\ 0 < \beta < 1- \tilde{c}_1/\cL }}  yx^{\beta-1}\\
&~\leq~ 
3 \sum_{\substack{\chi, ~\rho \neq \beta_1  \\ |\gamma| < T
	 \\ 1/2 \leq \beta < 1- \tilde{c}_1/\cL }}  yx^{\beta-1} 
~\leq~
 -3y \int_{1/2}^{1-\tilde{c}_1/\cL} x^{\sigma -1 } dN^*(\sigma, T),
 \end{split}
\end{equation*}
where 
$$
N^*(\sigma, T) ~=~ \sum_{\chi}\sum_{\substack{\rho \neq \beta_1
\\ \sigma < \beta < 1 \\ |\gamma| < T}} 1
$$
and $\tilde{c}_1$ is a positive constant (see \thmref{zerofreez}).
Let $c_2$ be a positive constant which is sufficiently large and  
$x \geq 2 Q^{4c_2}$. Also choose $T = Q^{-\frac{1}{n_E}} x^{\frac{1}{4c_2 n_E}}$.
Applying \thmref{ZerDen}, we obtain
\begin{equation}\label{PsiSh2}
\begin{split}
-\int_{1/2}^{1-\tilde{c}_1/\cL} x^{\sigma -1 } dN^*(\sigma, T) 
&~=~
x^{-1/2} N^*\(1/2, T\) 
~+~ \log x \int_{1/2}^{1-\tilde{c}_1/\cL}  
x^{\sigma -1 }  N^*(\sigma, T) ~ d\sigma \\
&~\ll~
 x^{-3/8} ~+~ e^{-3\tilde{c}_1 c_2}.
 \end{split}
\end{equation}
 We note that $D_L \geq D_E \cQ$ (see \cite[Lemma 4.2]{BS}) 
 and hence $Q = Q(L/E) \leq D_L n_E^{n_E}\leq D_L n_L^{n_L}$.
 Now we suppose that $x \geq \(D_L n_L^{n_L}\)^B$, 
 where $B= B(c_2)$ is a sufficiently large positive constant which depends on $c_2$. 
 Then we can check that
\begin{equation}\label{PsiSh3}
\cE_1 ~\ll~
 \frac{\#C}{\#G} \cdot x^{1- \frac{1}{8c_2 n_E}} 
\phantom{mm}\text{and}\phantom{mm}
\cE_2 ~\ll~
 \frac{\#C}{\#G} \cdot x^{1- \frac{1}{8c_2 n_E}} .
\end{equation}
We suppose that $y \geq x^{1- \frac{1}{16 c_2 n_E}}$. 
Now \thmref{Cheb-short} follows from \eqref{PsiSh1}, \eqref{PsiSh2} and \eqref{PsiSh3}. \qed

\section{Sato-Tate conjecture in short intervals}

\smallskip

\subsection{Proof of \thmref{STshomin}}
Suppose that GRH is true.
Let $M \geq 1$ be an integer and  
$I \subseteq [-1, 1]$ be a subset which can be $\text{Sym}^M$-minorized. 
Then there exist $b_0, ~b_1, \cdots , b_M \in \R$ with $b_0 > 0$ such that
$$
\mathds{1}_{I}(t) 
~\geq~ 
\sum_{n=0}^{M} b_n U_n(t)
~~\text{ for all~} t \in [-1,1].
$$
Hence we get
\begin{equation}\label{satint1}
\sum_{\substack{x < p \leq x+y \\ p \nmid N}}
\mathds{1}_{I}(\lambda_f(p)) \log p
~\geq~ 
\sum_{n=0}^{M} b_n \sum_{\substack{x < p \leq x+y \\ p \nmid N}} 
U_n(\lambda_f(p)) \log p.
\end{equation}
From \cite[Page 3596]{RT}, we have
$$
\bigg|\sum_{\substack{x < p \leq x+y \\ p \nmid N}} 
U_n(\lambda_f(p)) \log p \bigg|
 ~\ll~ x^{1/2} (\log x)^2
$$
for any $n \geq 1$. Here the implied constant depends on $M$ and $f$.
Note that the proof in \cite{RT} is given for non-CM newforms of square-free level but it
goes through also for non-CM forms of arbitrary level.
If $n=0$, we have (see \cite[page 113]{Da})
$$
\sum_{\substack{x < p \leq x+y \\ p \nmid N}} \log p 
~=~ y + O\(x^{1/2} (\log x)^2\).
$$
Hence from \eqref{satint1}, we get
\begin{equation*}
\sum_{\substack{x < p \leq x+y \\ p \nmid N}}
\mathds{1}_{I}(\lambda_f(p)) \log p
~\geq~ b_0 y + O\(x^{1/2} (\log x)^2\),
\end{equation*}
where the implied constant depends on $M$, 
$\max_{0 \leq i \leq M} |b_i|$ and $f$. 
This completes the proof of \thmref{STshomin}. \qed

\section{Large prime factors of Fourier coefficients in short intervals}

\smallskip
In this section, we detail the proofs of \thmref{thm1}, \thmref{thm2}, 
\thmref{thmn} and \rmkref{density}. We need the following 
lemmas to prove them.

\begin{lem}\label{red-p}
Let $n \ge 1$ be a natural number and $p \nmid N$
be a prime number. Then for $d | (n+1)$, we have
$$
P\(a_f(p^n)\) ~\geq~ P\(a_f(p^{d-1})\)
$$
provided $a_f(p^{n}) \ne 0$.
\end{lem}

\begin{proof}
For any prime $p \nmid N$ and integer $n \geq 1$, we have
$$
a_f(p^{n+1}) ~=~ a_f(p) a_f(p^n) - p^{k-1} a_f(p^{n-1}).
$$
Hence for natural numbers $n\ge 2$, we get
\begin{equation}\label{Lucaf}
a_f(p^{n-1}) ~=~ \frac{\alpha_p^n - \beta_p^n}{\alpha_p - \beta_p},
\end{equation}
where $\alpha_p, \beta_p$  are the roots 
of the polynomial $x^2-a_f(p)x + p^{k-1}$. 
Since $a_f(p)$'s are assumed to be integers, 
it follows from \eqref{Lucaf} that
$$
a_f(p^{d-1}) \mid a_f(p^n) \phantom{m}\text{whenever  } d \mid n+1
$$
provided $a_f(p^{d-1}) \neq 0$ 
(see \cite[page 37, Theorem IV]{RC} and \cite[page 434, Eq. 14]{St}). 
Hence if $a_f(p^{n}) \ne 0$, we obtain
$$
P\(a_f(p^n)\) ~\geq~ P\(a_f(p^{d-1})\)
$$
whenever $d \mid (n+1)$.
\end{proof}

\begin{lem}\label{red1}
Let $h(x)$ be a real valued non-negative function of $x$.
 Also let $q \ge 2$ be a prime number,
$V_q(x)~=~ \left\{ p \in (x,  ~x + h(x)] ~:~ p \nmid N,~ a_f(p^{q-1}) \neq 0 \right\}$
and
\begin{equation*}\label{pro-1}
\prod_{p \in V_q(x)} |a_f(p^{q-1})| 
~=~ \prod_{\ell ~\text{prime}} {\ell}^{\nu_{x, \ell}}.
\end{equation*}
Then we have
\begin{equation*}
\nu_{x, \ell} 
~\leq~
\sum_{1 \leq m \leq \frac{\log(q x^{qk})}{\log \ell}} 
\Big(\pi_{f,~q-1}(x + h(x), \ell^m) ~-~ \pi_{f,~q-1}(x, \ell^m) \Big).
\end{equation*}
\end{lem}

\begin{proof}
Note that, using Deligne's bound, we have
\begin{eqnarray*}
\nu_{x,\ell} 
~=~ 
\sum_{\substack{p \in V_q(x) }} \nu_\ell(a_f(p^{q-1}))\
&=&
\sum_{\substack{p \in V_q (x)}} 
\sum_{\substack{m \geq 1 \\ \ell^m | a_f(p^{q-1})}} 1 \\
&=&
\sum_{1 \leq m \leq \frac{\log(q x^{(q-1)(k-1)/2})}{\log \ell}} 
\sum_{\substack{p \in V_q(x) \\ a_f(p^{q-1}) \equiv 0 (\text{ mod } \ell^m)}} 1  \nonumber \\
&\le& 
\sum_{1 \leq m \leq \frac{\log(q x^{qk/2})}{\log \ell}}  
\bigg( \pi_{f, ~q -1}( x + h(x), ~\ell^m) - \pi_{f, ~q-1}(x,~ \ell^m)\bigg).
\end{eqnarray*}
\end{proof}

\subsection{Proof of \thmref{thm1}}
Let $f$ be as in \thmref{thm1} and $\e > 0$ be a real number. 

Applying \lemref{red-p}, we see that to prove \thmref{thm1}, 
it is sufficient to consider $n=q-1$, where $q$ is a prime number. 
The case $q=2$ corresponds to $n=1$ whereas when
$n>1$, we can assume that $q$ is an odd prime.

For any real number $A > 0$, set $\eta_1(x) = \frac{x}{(\log x)^A}$.
Let $V_q(x)$ be as in \lemref{red1} for $h(x) = \eta_1(x)$ and
\begin{equation}\label{prodafp}
\prod_{p \in V_q(x)} |a_f(p^{q-1})| 
~=~ \prod_{\ell ~\text{prime}} {\ell}^{\nu_{x, \ell}}.
\end{equation}
Then by \lemref{red1}, we have
\begin{equation}\label{vxq}
\nu_{x,\ell} 
~\le~
\sum_{1 \leq m \leq \frac{\log(q x^{ qk} )}{\log \ell}} 
\Big(\pi_{f,~q-1}(x + \eta_1(x), \ell^m) 
~-~ \pi_{f,~q-1}(x, \ell^m) \Big).
\end{equation}
From \thmref{pifshort}, there exists a constant $c > 0$ 
depending on $f$  and $A$ such that whenever
$1< \ell^m \leq c \frac{(\log x)^{1/4}}{(\log\log x)^{1/4}}$, 
we have
\begin{equation}\label{unpixd}
\pi_{f,~q-1}(x + \eta_1(x), \ell^m) ~-~ \pi_{f,~q-1}(x, \ell^m) 
~\ll~ \delta_{q-1}(\ell^m)~ \pi(\eta_1(x)) .
\end{equation}
Suppose that 
\begin{equation}\label{p-b}
P\(a_f(p^{q-1})\) ~\leq ~ (\log x^{q})^{1/8} (\log\log x^{q})^{3/8 -\e}
\end{equation}
for all $p \in V_q(x)$. Set
\begin{equation}\label{choice1}
z ~=~ c \frac{(\log x)^{1/4}}{(\log\log x)^{1/4}}
\phantom{mm}\text{and}
\phantom{mm}
w ~=~  (\log x^q)^{1/8} (\log\log x^q)^{3/8-\e}.
\end{equation}
From now on, assume that $\ell \le w$ and $x$ be sufficiently large.
For any prime $\ell  \leq w$,
set
$$
m_0 ~=~ m_0(x, \ell) ~=~  \Big[\frac{\log z}{\log \ell}\Big].
$$ 

\subsection*{When $n=1$ or equivalently $q=2$.}
Using \eqref{unpixd} and \lemref{deltal}, we obtain
\begin{eqnarray}\label{pifm01}
\sum_{1 \leq m \leq m_0} 
\( \pi_{f, ~1}(x + \eta_1(x), ~\ell^m) - \pi_{f, ~1}(x, ~\ell^m)\)
~\ll~ 
\sum_{1 \leq m \leq m_0} \delta(\ell^m) \pi(\eta_1(x))  
&\ll &	
\sum_{1\leq m \leq m_0} \frac{\pi(\eta_1(x))}{\ell^m}  \nonumber \\
&\ll& 
\frac{\pi(\eta_1(x))}{\ell} 
\end{eqnarray}
 and 
\begin{eqnarray}\label{pimm01}
\sum_{m_0 < m \leq \frac{\log(2 x^{2k} )}{\log \ell}} 
\( \pi_{f, ~1}(x+ \eta_1(x), ~\ell^m) - \pi_{f, ~1}(x,~ \ell^m)\)  
&\leq&
\delta(\ell^{m_0}) \pi(\eta_1(x))
\sum_{m \leq \frac{\log(2 x^{ 2k} )}{\log \ell}} 1 
\phantom{mmm} \nonumber\\
&\ll&
\frac{\pi(\eta_1(x))\log x}{\ell^{m_0} \log \ell}	
~\ll~ 
\frac{\eta_1(x)}{z} \cdot~\frac{\ell}{\log \ell}.
\end{eqnarray}

From \eqref{pifm01} and \eqref{pimm01}, we deduce that
\begin{equation}\label{pifs1}
\nu_{x,\ell} ~\ll~ 	\frac{\eta_1(x)}{z} \cdot~\frac{\ell}{\log \ell}.
\end{equation}
It follows from \eqref{prodafp}, \eqref{p-b}, \eqref{choice1} and \eqref{pifs1} that 

\begin{equation}\label{uppsum1}
\begin{split}
\sum_{p \in V_2(x)} \log|a_f(p)|
~=~
\sum_{\ell \leq w} \nu_{x, \ell} \log \ell 
~\ll~  
\frac{\eta_1(x)}{z} \cdot \sum_{\ell \leq w} \ell 
~\ll~ 
\frac{\eta_1(x)}{z} \cdot \frac{w^2}{\log w} 
~\ll~ 
\frac{\eta_1(x)}{(\log\log x)^{\e}}
\end{split}
\end{equation}
for all sufficiently large $x$.
Applying \thmref{STshort} with $I = [-1, -1/2]$ and \red{$M=4$}, we get
$$
\sum_{\substack{p \in (x, ~x + \eta_1(x)] \\ \lambda_f(p) \in I}} \log p 
	~\gg~ \eta_1(x)
$$
for all sufficiently large $x$. Hence we deduce  that
\begin{equation}\label{eqRl1}
\sum_{p \in V_2(x)} \log|a_f(p)| 
~\geq~ 
\sum_{\substack{p \in (x,  x + \eta_1(x)] \\ \lambda_f(p) \in I}} \log|a_f(p)| 
~\gg~ 
\sum_{\substack{p \in (x,  x + \eta_1(x)] \\ \lambda_f(p) \in I}} \log p 
~\gg~ \eta_1(x)
\end{equation}
for all sufficiently large $x$. This is a contradiction 
to \eqref{uppsum1} when $x$ is sufficiently large
and completes the proof when $n=1$.

\subsection*{When $n >1$ or equivalently $q$ is an odd prime.}
Let $\ell \le w$ be a prime such that $\ell \equiv 0, \pm 1 ~(\text{mod } q)$. 
Then for such an $\ell$, it follows from \lemref{deq-1l}, \lemref{dq-1lm} 
and \eqref{unpixd} that 
\begin{equation}\label{vxq1}
\sum_{1 \leq m \leq m_0}  
\Big(\pi_{f,~q-1}(x + \eta_1(x), \ell^m) 
~-~ \pi_{f,~q-1}(x, \ell^m) \Big)
~\ll~ 
q\sum_{1\leq m \leq m_0}  \frac{ \pi(\eta_1(x))}{\ell^m}  
~\ll~ 
q \cdot \frac{\pi(\eta_1(x))}{\ell} 
\end{equation}
and
\begin{eqnarray}\label{vxq2}
&&
\sum_{m_0 < m \leq \frac{\log(q x^{q k} )}{\log \ell}} 
\Big(\pi_{f,~q-1}(x + \eta_1(x), \ell^m) 
~-~ \pi_{f,~q-1}(x, \ell^m) \Big)  \nonumber\\
&\leq&
\sum_{m \leq \frac{\log(q x^{q k} )}{\log \ell}} 
\Big(\pi_{f,~q-1}(x + \eta_1(x), \ell^{m_0}) - \pi_{f,~q-1}(x, \ell^{m_0}) \Big)
\phantom{mmm} \nonumber\\
&\ll&
\frac{q}{\ell^{m_0}}  \pi(\eta_1(x)) \cdot\frac{q\log x}{\log \ell}
~\ll~
\frac{q^2 \eta_1(x)}{\ell^{m_0} \log \ell}
~\ll~ 
\frac{q^2 \eta_1(x)}{z} \cdot~\frac{ \ell}{\log \ell}.
\end{eqnarray}
From \eqref{vxq1} and \eqref{vxq2}, we get
\begin{equation}\label{vxq3}
\nu_{x,\ell} 
~\ll~ \frac{q^2 \eta_1(x)}{z} \cdot \frac{\ell}{\log \ell}.
\end{equation}
Note that if $\ell \not\equiv 0, \pm 1 ~(\text{mod } q)$, 
we have $C_{\ell^m, ~q-1} = \emptyset$ 
(see \lemref{deq-1l} and \lemref{dq-1lm}).
Hence if $\ell^m \mid a_f(p^{q-1})$, 
then we must have $p \mid \ell N$ (see section \ref{Smod}). 
Since $p \in V_q(x)$, we obtain $p = \ell$. 
Hence we have 
$\nu_{x, \ell} \leq \nu_{\ell}(a_f(\ell^{q-1})) \ll  kq$ 
if  $\ell \not\equiv 0, \pm 1 ~(\text{ mod } q)$.
It follows from \eqref{prodafp}, \eqref{p-b} and \eqref{choice1}  that
\begin{equation}\label{log=vxll}
\sum_{p \in V_q(x)} \log|a_f(p^{q-1})| 
~=~
\sum_{\ell \leq w} \nu_{x, \ell} \log \ell.
\end{equation}
Now applying \eqref{vxq3} and Brun-Titchmarsh inequality 
(see \cite[Theorem~3.8]{HalR}), we obtain
\begin{equation*}\label{upp}
\begin{split}
\sum_{\substack{\ell \leq w \\ \ell \equiv 0, \pm 1 (\text{ mod } q)}}
\nu_{x,\ell} \log \ell
&~\leq~ \nu_{x,q} \log q 
~+~  \sum_{\substack{\ell \leq w \\ \ell \equiv \pm 1 (\text{ mod } q)}}
\nu_{x,\ell} \log \ell \\
& ~\ll~
\frac{q^3 \eta_1(x)}{z} ~~+~
\frac{q^2 \eta_1(x)}{z}
\sum_{\substack{\ell \leq w \\\ell \equiv \pm 1 (\text{ mod } q)}} \ell 
~~\ll~~ 
\frac{q^3 \eta_1(x)}{z} ~+~ \frac{q^2 \eta_1(x)}{z}  \frac{w^2}{q \log (w/q)} 
\end{split}
\end{equation*}
for all  sufficiently large $x$ depending on $ A, \e, q$ and $f$.
Also we have
\begin{equation*}
\begin{split}
\sum_{\substack{\ell \leq w \\ \ell \not\equiv 0, \pm 1(\text{ mod } q)}}
\nu_{x,\ell} ~\log \ell 
~~\ll~~
q \sum_{\ell \leq w} \log \ell 
~\ll~
q w.
\end{split}
\end{equation*}
Hence we conclude that
\begin{equation}\label{vxllog}
\sum_{\ell \leq w} \nu_{x, \ell} \log \ell
~\ll~
\frac{q^3 \eta_1(x)}{z} 
~+~ \frac{q^2 \eta_1(x)}{z}  \frac{w^2}{q \log (w/q)} 
\end{equation}
for all  sufficiently large $x$ depending on $A, \e, q, f$ and the 
implied constant depends only on $f$. 

Using Deligne's bound, we can write
$$
a_f(p) ~=~ 2 p^{\frac{k-1}{2}} \lambda_f(p)~,
\phantom{m} 
\lambda_f(p) \in [-1, 1].
$$
For any prime $p \nmid N$, we can deduce from \eqref{Lucaf} that
\begin{equation}\label{afPsi}
\begin{split}
a_f(p^{q-1}) 
~=~ \prod_{j=1}^{\frac{q-1}{2}} 
\( a_f(p)^2 - 4 \cos^2(\pi j/q) p^{k-1}\) 
~=~ (4 p^{k-1})^{\frac{q-1}{2}}  \prod_{j=1}^{\frac{q-1}{2}}
\(\lambda_f(p)^2 - \cos^2(\pi j/q)\).
\end{split}
\end{equation}
Set
\begin{equation}\label{Def-I}
\cI_{q} ~=~ \bigg\{ t \in [-1, 1] 
~:~ \bigg|t-\cos\(\frac{\pi j}{q}\)\bigg| \geq \frac{1}{q^2} 
~\text{ and }~ 
\bigg|t+\cos\(\frac{\pi j}{q}\)\bigg| \geq \frac{1}{q^2} 
~\phantom{m} \forall~ 1 \leq j \leq \frac{q-1}{2} \bigg\}
\end{equation}
From \rmkref{SymMin}, the set $\cI_q$ can be $\text{Sym}^M$-minorized 
if $M$ is sufficiently large (depending on $q$) 
and hence from \thmref{STshort}, we deduce that
\begin{equation}\label{lambdI}
\sum_{\substack{p \in (x, x+\eta_1(x)] \\ \lambda_f(p) \in \cI_q}} \log p 
~\gg~ \eta_1(x),
\end{equation}
where the implied constant depends on $q$ and $f$.
For any prime $p \in V_q(x)$ with $\lambda_f(p) \in \cI_q$, 
we have $|a_f(p^{q-1})| \geq (4p^{k-1})^{\frac{q-1}{2}} {q}^{-2(q-1)}$. 
Thus from \eqref{lambdI}, we get
\begin{equation}\label{Lpwx0}
\begin{split}
\sum_{p \in V_{q}(x)} \log|a_f(p^{q-1})| 
~\geq ~
\sum_{\substack{p \in V_q(x)\\ \lambda_f(p) ~\in~ \cI_{q}}} 
\log|a_f(p^{q-1})| 
&~\gg~
 \sum_{\substack{p \in V_q(x)\\ \lambda_f(p) ~\in~ \cI_{q}}} \log p 
 ~+~ O\(  \pi(\eta_1(x))\) \\
&~\gg~ \eta_1(x)
\end{split}
\end{equation}
for all sufficiently large $x$ depending on $A, q$ and $f$. 
Here the implied constant depends on $q$ and $f$. 
This gives a contradiction to \eqref{vxllog} 
if $x$ is sufficiently large depending on $A, \e, q$ and $f$.
This completes the proof of  \thmref{thm1}. 
 \qed

\subsection{Proof of \thmref{thm2}}\label{SecJq}
Suppose that GRH is true.
The proof now follows along the lines of the proof of \thmref{thm1}.
As in the proof of \thmref{thm1}, it is sufficient to investigate large prime factors 
$a_f(p^{q-1})$, where $q$ is a prime number. 
For any real number $\e \in (0, 1/10)$ and for any prime $q$,
let $V_q(x)$ be as in \lemref{red1} for $h(x)~=~x^{1/2+\e}$.

When $q$ is an odd prime, set 
\begin{equation*}
\cJ_{q} ~=~ \bigg\{ t \in [-1, 1] 
~:~ \bigg|t-\cos\(\frac{\pi j}{q}\)\bigg| \geq \frac{1}{ Cq^2} 
~\text{ and }~ \bigg|t+\cos\(\frac{\pi j}{q}\)\bigg| \geq \frac{1}{Cq^2} 
\phantom{m} \forall~ 1 \leq j \leq (q-1)/2\bigg\},
\end{equation*}
where $C > 0$ is a constant such that $\mu_{ST}\(\cJ_q\) > 1/2$.
From \rmkref{SymMin}, we know that $\cJ_q$ is 
$\text{Sym}^M$-minorized (with $b_0=b_0 (q) > 0$) 
if $M$ is sufficiently large (depending on $q$). Let  $0 < b < \min\{b_0, \frac{1}{7}\}$ 
and $c$ be a positive constant which will be chosen later.
When $q=1$, $c_1$ is a positive constant which will be chosen later.

Suppose that for any $p \in V_q(x)$, 
$$
P(a_f(p^{q-1})) ~\leq~ w,
$$
where
$$
w ~=~
\begin{cases}
	c_1x^{\e/7} (\log x)^{2/7}  & \text{when $q=2$}; \\
	c x^{\e b} & \text{when $q$ is an odd prime}.
\end{cases}
$$
Write
\begin{equation*}
\prod_{p \in V_q(x)} |a_f(p^{q-1})| 
~=~ \prod_{\ell \text{ prime }} \ell^{\nu_{x,\ell}}.
\end{equation*}
This implies that
\begin{equation}\label{req1}
\sum_{p \in V_q(x)} \log|a_f(p^{q-1})| 
~=~ \sum_{\ell \leq w} \nu_{x,\ell} \log \ell,
\end{equation}
Using  \lemref{red1}, we know that
\begin{equation}
\nu_{x,\ell} ~\leq~ 
\sum_{1 \leq m \leq \frac{\log(q x^{q k })}{\log \ell}} 
 \bigg(\pi_{f, q-1}(x+ x^{1/2+\e}, \ell^m) - \pi_{f, q-1}(x, \ell^m)\bigg).
\end{equation}
Set 
\begin{equation*}
z ~=~
\begin{cases}
c_1\frac{x^{2\e/7}}{(\log x)^{3/7}}  & \text{when $q=2$}; \\
c \frac{x^{2\e b}}{\log x} & \text{when $q$ is an odd prime}~~.
\end{cases}
\end{equation*}
From now on, assume that $\ell \le w$, $x$ be sufficiently large
and $m_0 = m_0(x, \ell) = \Big[\frac{\log z}{\log \ell}\Big]$.

\subsection*{When $n=1$ or equivalently $q=2$.}
Applying \thmref{pifGRH} and \lemref{deltal}, we have 
\begin{equation}\label{pifm0}
\begin{split}
\sum_{1 \leq m \leq m_0}  
\left( \pi_{f, ~1}(x+ x^{1/2+\e}, \ell^m) - \pi_{f, ~1}(x, \ell^m)\right)
& ~=~ \sum_{1 \leq m \leq m_0} \delta(\ell^m) 
\bigg\{\frac{x^{1/2+\e}}{\log x}
~+~ O\(\ell^{4m} x^{1/2} \log x \) \bigg\}\\
& = \frac{1}{\ell} \cdot \frac{x^{1/2+\e}}{\log x} 
~+~ O\(\frac{x^{1/2+\e}}{\ell^2 \log x}\) 
~+~ O\(z^3 x^{1/2} \log x\).
\end{split}
\end{equation}
Further 
$\sum_{m_0 < m \leq \frac{\log(2 x^{2k})}{\log \ell}}  
\left(\pi_{f, ~1}(x+ x^{1/2+\e}, \ell^m) -  \pi_{f, ~1}(x, \ell^m)\right)$
is less than or equal to
\begin{eqnarray}\label{pifmm0}
\bigg(\pi_{f, ~1}(x+ x^{1/2+\e}, \ell^{m_0}) - \pi_{f, ~1}(x, \ell^{m_0})\bigg) 
\sum_{m \leq \frac{\log(2 x^{2k})}{\log \ell}} 1 
&\ll&
\(\frac{x^{1/2+\e}}{\ell^{m_0} \log x} 
~+~ \ell^{3m_0} x^{1/2} \log x \)  \frac{\log x}{\log \ell} \nonumber \\
&\ll&
\frac{x^{1/2+\e}}{z}\frac{\ell}{\log \ell} 
~+~ z^3 x^{1/2}\frac{(\log x)^2}{\log \ell}.
\end{eqnarray}
From \eqref{pifm0} and \eqref{pifmm0}, we get
\begin{equation}\label{pifvxq}
\nu_{x,\ell} ~\leq~
\frac{1}{\ell} \cdot \frac{x^{1/2+\e}}{\log x}  
~+~ O\(\frac{x^{1/2+\e}}{\ell^2 \log x}  
~+~  \frac{x^{1/2+\e}}{z}\frac{\ell}{\log \ell}
~+~  z^3 x^{1/2}\frac{(\log x)^2}{\log \ell}  \).
\end{equation}
It follows from \eqref{pifvxq} that
\begin{equation*}
\begin{split}
\sum_{\ell \leq w} \nu_{x,\ell} \log \ell 
&~\leq~ 
\frac{x^{1/2+\e}}{\log x} \log w 
~+~ c_5\(\frac{x^{1/2+\e}}{\log x} 
~+~ \frac{x^{1/2+\e}}{z}\frac{w^2}{\log w} 
~+~ z^3 x^{1/2} (\log x)^2 \frac{w}{\log w}\),
\end{split}
\end{equation*}
where $c_5>0$ is a constant depending on $f$. Now we choose $c_1$ such that
$2000 c_5 c_1(1 + c_1^3) < \e$.
Then by substituting the values of $w$ and $z$, we obtain
\begin{equation}\label{uppRH}
\sum_{\ell \leq w} \nu_{x,\ell} \log \ell 
~<~ \frac{x^{1/2+\e}}{30} 
\end{equation}
for all sufficiently large $x$  depending on $\e$ and $f$. On the other hand, 
from \thmref{STshomin} with $I = [-1, -0.1]$ and \rmkref{SymMin}, we get
$$
\sum_{\substack{ x < p \leq x+ x^{1/2+\e} \\ \lambda_f(p) \in I}} \log p 
~>~ \frac{2}{25} x^{1/2+\e}
$$
for all sufficiently large $x$ depending on $\e$ and $f$.
Hence we deduce that
\begin{equation}\label{lowRH}
\sum_{p \in V_2(x)} \log|a_f(p)| 
~>~ \frac{3}{40}\frac{k-1}{2} \cdot x^{1/2+\e}
\end{equation}
for all sufficiently large $x$ depending on $\e$ and $f$. 
This is a contradiction to \eqref{uppRH}. 

\subsection*{When $n >1$ or equivalently $q$ is an odd prime.}
Arguing as before and applying \thmref{STshomin} to the
interval $\cJ_q$, we can show that
\begin{equation}\label{lb}
\sum_{p \in V_q(x)} \log |a_f(p^{q-1})| 
~~\ge~~
 \frac{k q b}{16} x^{1/2 + \e}
\end{equation}
for all sufficiently large $x$ depending on $\e, q$ and $f$.

For any prime $\ell \leq w$ with $\ell \not\equiv 0, \pm 1 ~(\text{ mod } q)$,
we can deduce as in \thmref{thm1} that
\begin{equation}\label{n-nu}
\nu_{x, \ell} ~\leq~ \nu_{\ell} (a_f(\ell^{q-1})) ~=~ O(kq).
\end{equation}
When $\ell \equiv 0, \pm 1 (\text{mod } q)$,
applying \lemref{deltal}, \lemref{deq-1l}, \lemref{dq-1lm}
and \thmref{pifGRH}, we have 
$\sum_{1 \leq m \leq m_0}  
\left( \pi_{f,~q-1}(x+ x^{1/2+\e}, \ell^m) -  \pi_{f,~q-1}(x, \ell^m)\right)$
is less than or equal to
\begin{equation}\label{n-eq1}
 \frac{q-1}{2\ell} \cdot \frac{x^{1/2+\e}}{\log x}  
 ~~+~~O\(\frac{qx^{1/2+\e}}{\ell^2 \log x}\) ~+~ O\(q z^3 x^{1/2} \log x\).
\end{equation}
Further 
$\sum_{m_0 < m \leq \frac{\log(q x^{qk})}{\log \ell}}  
\left(\pi_{f, q-1}(x+ x^{1/2+\e}, \ell^m) - \pi_{f, q-1}(x, \ell^m)\right)$
is less than or equal to
\begin{eqnarray}\label{n-eq2}
\bigg(\pi_{f, q-1}(x+ x^{1/2+\e}, \ell^{m_0}) - \pi_{f, q-1}(x, \ell^{m_0})\bigg) 
\sum_{m \leq \frac{\log(q x^{ q k} )}{\log \ell}} 1 
&\ll&
\(\frac{ q x^{1/2 + \e} }{\ell^{m_0} \log x} 
~+~ q \ell^{3m_0}  x^{1/2} \log x \) \frac{q\log x}{\log \ell} \nonumber \\
&\ll&
\frac{q^2 x^{1/2+\e}}{z}\frac{\ell}{\log \ell} 
~+~ q^2 z^3 x^{1/2}\frac{(\log x)^2}{\log \ell}.
\end{eqnarray}
From \eqref{n-nu}, \eqref{n-eq1} and \eqref{n-eq2}, we get
\begin{equation}\label{pifvxq1}
\nu_{x,\ell} ~\leq~
\frac{q-1}{2\ell} \cdot \frac{x^{1/2+\e}}{\log x}  
~+~ O\(\frac{qx^{1/2+\e}}{\ell^2 \log x}  
~+~  \frac{q^2x^{1/2+\e}}{z}\frac{\ell}{\log \ell}
~+~  q^2z^3 x^{1/2}\frac{(\log x)^2}{\log \ell}  \).
\end{equation}
It follows from \eqref{req1} and \eqref{pifvxq1} that
\begin{equation*}
\begin{split}
\sum_{\ell \leq w} \nu_{x,\ell} \log \ell 
&~\leq~ 
\frac{x^{1/2+\e}}{\log x} \log w 
~+~ c_6\(\frac{qx^{1/2+\e}}{\log x} 
~+~ \frac{qx^{1/2+\e}}{z}\frac{w^2}{\log w} 
~+~ q z^3 x^{1/2} (\log x)^2 \frac{w}{\log w}\) 
\end{split}
\end{equation*}
for all sufficiently large $x$ depending on $\e, q$ and $f$
and where $c_6>0$ is a constant depending on $\e, q$ and $f$.
Substituting the values of $z, w$ and by choosing $c$ such that
$2000 \cdot c_6 c < \e b^2$, we get a contradiction to \eqref{lb}
for all sufficiently large $x$ depending on $\e, q$ and $f$.
Hence there exists a prime $p \in (x, x+x^{1/2+\e}]$ with $p \nmid N$ such that
$$
P\(a_f(p^{q-1})\) ~>~ c x^{\e b}
$$
for some positive constant $c$ depending on $\e, q, f$ 
and for all sufficiently large $x$ depending on $\e, q$ and $f$.
This completes the proof of \thmref{thm2}. \qed

\smallskip

\subsection{Proof of \thmref{thmn}}

Proof of this theorem follows along the lines of the proof of \thmref{thm1}
or \thmref{thm2}. Let $V_{q}(x)$ be as in \lemref{red1} with
$h(x)=\eta(x) = x^{3/4} \log x \cdot \log\log x$. Set
\begin{equation*}
z ~=~ c x^{1/14}\frac{(\log\log x)^{2/7}}{(\log x)^{1/7}} 
\phantom{mm}\text{and}\phantom{mm} 
w ~=~ c x^{1/28} (\log x)^{3/7} (\log\log x)^{1/7},
\end{equation*}
where $c> 0$ is a constant which will be chosen later. 
Suppose that
$$
P(a_f(p^{q-1})) ~\leq~ c x^{1/28} (\log x)^{3/7} (\log\log x)^{1/7}
$$
for any $p \in V_q(x)$. Write
$$
\prod_{p \in V_q(x)} |a_f(p^{q-1})| 
~=~ \prod_{\ell \leq w} \ell^{\nu_{x,\ell}}.
$$
Then
\begin{equation}\label{l=vxl}
\sum_{p \in V_q(x)} \log|a_f(p^{q-1})| 
~=~ \sum_{\ell \leq w} \nu_{x,\ell} \log \ell,
\end{equation}
where, using \eqref{red-p}, we see that
\begin{equation*}
\nu_{x, \ell} 
~\leq~
\sum_{1 \leq m \leq \frac{\log(q x^{qk})}{\log \ell}} 
\Big(\pi_{f,~q-1}(x + \eta(x), \ell^m) ~-~ \pi_{f,~q-1}(x, \ell^m) \Big).
\end{equation*}
Fix a prime $\ell \leq w$ such that $\nu_{x, \ell} \neq 0$. 
If $\ell \not\equiv 0, \pm 1 ~(\text{ mod } q)$, then as before, 
we have
$$
\nu_{x, \ell} \leq \nu_{\ell} (a_f(\ell^{q-1})) ~=~ O(kq).
$$
Now suppose that $\ell \equiv 0, \pm 1 (\text{mod } q)$ and 
set  $m_0 = \Big[\frac{\log z}{\log \ell}\Big]$.
Let $x$ be sufficiently large from now on.
Then applying \thmref{pifGRH}, \lemref{deq-1l} and \lemref{dq-1lm}, we get 
\begin{equation}\label{pif1}
\begin{split}
\sum_{1 \leq m \leq m_0}  
\Big(\pi_{f,~q-1}(x + \eta(x), \ell^m) - \pi_{f,~q-1}(x, \ell^m) \Big)
& ~\leq~
\frac{q-1}{2 \ell}\frac{\eta(x)}{\log x} 
~+~ 
O\(\frac{ q}{\ell^2} \frac{\eta(x)}{\log x}\) 
~+~ 
O\( q z^3 x^{1/2} \log x\)
\end{split}
\end{equation}
and
\begin{equation}\label{pif2}
\begin{split}
\sum_{m_0 < m \leq \frac{\log(q x^{qk})}{\log \ell}}  
\Big(\pi_{f,~q-1}(x + \eta(x), \ell^m) ~-~ \pi_{f,~q-1}(x, \ell^m) \Big)
&~\ll~
\frac{q^2 \eta(x)}{z}\frac{\ell}{\log \ell} 
~+~ 
q^2 z^3 x^{1/2}\frac{(\log x)^2}{\log \ell}.
\end{split}
\end{equation}
From \eqref{pif1} and \eqref{pif2}, we get
\begin{equation}\label{pif3}
\nu_{x,\ell} 
~\leq~
\frac{q-1}{2 \ell}\frac{\eta(x)}{\log x} 
~+~ O\(\frac{ q}{\ell^2} \frac{\eta(x)}{\log x}  
~+~  \frac{q^2 \eta(x)}{z}\frac{\ell}{\log \ell}
~+~  q^2 z^3 x^{1/2}\frac{(\log x)^2}{\log \ell}\)
\end{equation}
when $\ell \equiv 0, \pm 1 (\text{mod } q)$.
It follows from \eqref{pif3} and Brun-Titchmarsh inequality that
\begin{equation}\label{1pm1q}
\begin{split}
\sum_{\substack{\ell \leq w \\ \ell \equiv \pm 1(\text{ mod } q)}} 
\nu_{x,\ell} \log \ell 
~\leq~  
\frac{\eta(x)}{\log x}  \log w 
~+~ c_7\( q \frac{\eta(x)}{\log x}  
~+~ \frac{q \eta(x)}{z}\frac{w^2}{\log (w/q)} 
~+~  q z^3 x^{1/2} (\log x)^2 \frac{w}{\log (w/q)}\)
\end{split}
\end{equation}
for all sufficiently large $x$ (depending on $q$ and $f$). 
Here $c_7$ is a positive constant depending only on~$f$. 
We also have
\begin{equation}\label{vxqn01}
\begin{split}
\nu_{x,q} \log q 
~\ll~
\frac{\eta(x)}{\log x}  \log q 
~+~ \frac{q^3 \eta(x)}{z} 
~+~ q^2 z^3 x^{1/2} (\log x)^2 
~~~\text{  and }
\sum_{\substack{\ell \leq w \\ \ell \not\equiv 0, \pm 1 (\text{mod } q)}} 
\nu_{x, \ell}\log \ell  
~\ll~ q w.
\end{split}
\end{equation}
Let $c$ be such that $2000 \cdot c_7 (c+c^4) < 1$. 
Then by substituting the values for $w$ and $z$ 
in \eqref{1pm1q} and \eqref{vxqn01}, we deduce that
\begin{equation}\label{vx113.5}
\sum_{\ell \leq w} \nu_{x, \ell} \log \ell 
~~<~ \frac{q}{20}~ \eta(x)
\end{equation}
for all sufficiently large $x$ (depending on $q$ and $f$).
Set $\cJ_q$ as in subsection \ref{SecJq}.
As before, by applying \thmref{EfSTGRH}, we can show that
\begin{equation}\label{lSq}
\sum_{p \in V_q(x)} 
\log |a_f(p^{q-1})|
~>~
\frac{kq}{17} \eta(x)
\end{equation} 
for all sufficiently large $x$ depending on $q$ and $f$. This gives 
a contradiction to \eqref{vx113.5} and completes the proof for
large prime factor of $a_f(p^{q-1})$ in the interval 
$(x, x+ \eta(x)]$ under GRH.

\subsection{Proof of \rmkref{rmk2}}
In \thmref{thm1}, instead of working with 
$$
V_q(x)~=~ \left\{ p \in (x, x+\eta_1(x)]  
~:~ p \nmid N,~ a_f(p^{q-1}) \neq 0 \right\},
$$
one has to consider 
$$
S_q(x)~=~ 
\left\{ p \in (x, x+\eta_1(x)]  ~:~ p \nmid N,~ a_f(p) \neq 0, 
~ P(a_f(p^{q-1})) ~\leq~ (\log x^q)^{1/8} (\log\log x^q)^{3/8 -\e}
\right\}.
$$
Arguing as in the proof of \thmref{thm1} 
(see \eqref{uppsum1}, \eqref{eqRl1}, \eqref{vxllog} and \eqref{Lpwx0}), 
we can deduce that
$$
\sum_{p \in S_q(x) \atop \lambda_f(p) \in \cI_q} \log p 
~+~ O\(\pi(\eta_1(x))\) 
~\ll~\sum_{p \in S_q(x)} \log|a_f(p^{q-1})| 
~=~
\sum_{\ell \leq w} \nu_{x, \ell} \log \ell
~\ll~\frac{\eta_1(x)}{(\log\log x)^{\e}}.
$$
Let $T_q(x) = \{p \in (x, x+\eta_1(x)] : \lambda_f(p) \in \cI_q\}$, 
where $\cI_2 = [-1, -1/2]$ and $\cI_q$ is as in \eqref{Def-I} for $q \ge 3$. 
Thus we get
\begin{equation}\label{eqR2}
\# \(S_q(x) \cap T_q(x) \) ~\ll~ 
\frac{\pi( \eta_1(x) )}{ (\log \log x)^{\e}}
\end{equation}
for all sufficiently large $x$ depending on $A, \e, q$ and $f$.  
As observed earlier, from \thmref{STshort}, 
there exists a positive constant $0 < b_1 < 1$ (depending on $A,q, f$) such that
\begin{equation}\label{eqR1}
\#T_q(x) ~\geq~ b_1 \pi( \eta_1(x) )
\end{equation}
for all sufficiently large $x$ (depending on $A, q$ and $f$). 
From \eqref{eqR2} and \eqref{eqR1}, we deduce that
$$
\limsup_{x \to \infty} \frac{\#S_q(x)}{\pi( \eta_1(x) )} 
~\leq~ 1-b_1 ~<~ 1.
$$
Thus there exists a positive constant $a_1$ such that
for all sufficiently large $x$,
there are at least $a_1\pi( \eta_1(x) )$ many primes
$p \in (x,  ~x + \frac{x}{(\log x)^A}]$ for which
\thmref{thm1} is true.

In \thmref{thm2}, we suppose that $\e > 0$ is sufficiently small 
and let $S_q(x)$ be the set of primes $ p \in ( x, x+x^{\frac{1}{2}+\e}]$ 
such that $a_f(p^{q-1}) \neq 0$ and
\begin{equation*}
P(a_f(p^{q-1})) ~\leq~
\begin{cases}
c_1x^{\e/7} (\log x)^{2/7}  & \text{when $q=2$}; \\
c x^{\e b} & \text{when $q$ is an odd prime},
\end{cases}
\end{equation*}
where $c, c_1$ are as in the proof of \thmref{thm2} 
and $T_q(x) = \{p \in (x, x+x^{\frac{1}{2}+\e}] : \lambda_f(p) \in \cI_q\}$. 
As before,
there exists a positive constant $0 < b_2 < 1$ (depending on $\e, f$) such that
\begin{equation}
\#T_q(x) ~\geq~ b_2 \frac{x^{1/2+\e}}{\log x}
\end{equation}
for all sufficiently large $x$ (depending on $\e, q$ and $f$).
In the case when $n=1$ or $q=2$, 
by choosing the constant $c_5> 0$ sufficiently small, 
we can deduce that
$$
\sum_{\ell \leq w} \nu_{x,\ell} \log \ell 
~<~ \frac{b_2}{2}  \cdot x^{1/2+\e}
$$ 
for all sufficiently large $x$ (depending on $\e, q$ and $f$).
Arguing as before, we can deduce that
$$
\limsup_{x \to \infty} \frac{\#S_q(x)}{x^{1/2+\e}/ \log x} 
~\leq~ 1-\frac{b_2}{2} ~<~ 1.
$$
One can deduce a similar conclusion in the remaining case of \thmref{thm2}. 
To deduce a similar conclusion for \thmref{thmn}, we proceed as follows.
Let  
$$
S_q(x) ~=~ \{ p \in (x, x+ \eta(x)] ~:~ a_f(p^{q-1}) \neq 0,~~ 
P(a_f(p^{q-1})) ~\leq~ c x^{1/28} (\log x)^{3/7} (\log\log x)^{1/7} \},
$$
where $\eta(x) = x^{3/4} \log x \cdot \log\log x$ 
and $c$ is as in the proof of \thmref{thmn}. 
Arguing as in the proof of \thmref{thmn}, 
we deduce that
\begin{equation}
\sum_{\ell \leq w} \nu_{x, \ell} \log \ell 
~<~ \frac{q}{20}~ \eta(x)
\end{equation}
for all sufficiently large $x$ (depending on $q$ and $f$).
Set $\cJ_q$ as in subsection \ref{SecJq} 
and $C>0$ is a constant which we will choose later.
Then we can deduce that
\begin{equation}
\sum_{p \in S_q(x)} 
\log |a_f(p^{q-1})|
~\geq~ 
\sum_{p \in S_q(x) \atop \lambda_f(p) \in \cJ_q} \log |a_f(p^{q-1})|
~\geq~
\frac{kq}{17} \cdot  \#\(S_q(x) \cap T_q(x)\) \log x + O\(\pi(\eta(x))\)
\end{equation} 
for all sufficiently large $x$ depending on $q$ and $f$. 
Here $T_q(x) = \{p \in (x, x+\eta(x)] : \lambda_f(p) \in \cJ_q\}$. 
Hence we get
$$
\limsup_{x \to \infty} \frac{ \#\(S_q(x) \cap T_q(x)\)}{\pi(\eta(x))} 
~\leq~ \frac{17}{20} < 1.
$$
We choose $C>0$ sufficiently large such that 
$\mu_{ST}(\cJ_{q}) > 1- \delta$ for some $  0<\delta < 17/2000$. 
Then we get
$$
\#T_q(x) ~\ge~ \(1- 2\delta\) \pi(\eta(x))
$$
for all sufficiently large $x$ depending on $q$ and $f$. 
Note that
\begin{equation}
\begin{split}
\#\(S_q(x) \cap T_q(x)\) 
&~=~ \#S_q(x) ~+~ \#T_q(x) ~-~ \#\(S_q(x) \cup T_q(x)\) \\
&~\geq~ \#S_q(x) ~+~ \(1- 2 \delta\) \pi(\eta(x)) ~-~ \pi(\eta(x)) \\ 
&~\geq~ \#S_q(x) ~-~ 2 \delta \pi(\eta(x)).
\end{split}
\end{equation}
Hence we deduce that 
$$
\limsup_{x \to \infty} \frac{\#S_q(x)}{\pi(\eta(x))} 
~\leq~ \frac{17}{20} ~+~ 2 \delta ~<~ 1.
$$

\qed

\medskip

\noindent
{\bf Acknowledgments.}
The authors would like would like to acknowledge the support of DAE number
theory plan project. The second author would like to thank 
the Institue of Mathematical Sciences, India and Queen's University, Canada 
for providing excellent atmosphere to work. The authors would like to thank 
the referee for suggesting us to estimate the  density of the set of primes 
for which Theorem 2 and Theorem 3 are true. Also, the authors would like to thank 
the referee for other valuable suggestions.

\end{document}